\documentclass[12pt]{article}
\usepackage[english]{babel}
\usepackage{latexsym,amssymb,amsmath,amsthm,amsfonts}
\usepackage{graphicx}

\def\'#1{\ifx#1i{\accent"13 \i}\else{\accent"13 #1}\fi}

\newtheorem{theorem}{Theorem}[section]
\newtheorem{corollary}[theorem]{Corollary}
\newtheorem{lemma}[theorem]{Lemma}

\title{On the 4-girth-thickness of the line graph of the complete graph}
\author{Christian Rubio-Montiel \footnotemark[1]}

\begin{document}
\maketitle
\def\thefootnote{\fnsymbol{footnote}}
\footnotetext[1]{Divisi{\' o}n de Matem{\' a}ticas e Ingenier{\' i}a, FES Acatl{\' a}n, Universidad Nacional Aut{\' o}noma de M{\' e}xico, 53150, Naucalpan, Mexico, {\tt christian.rubio@apolo.acatlan.unam.mx}.}

\renewcommand{\thefootnote}{\arabic{footnote}}

\begin{abstract}
The $g$-girth-thickness $\theta(g,G)$ of a graph $G$ is the minimum number of planar subgraphs of girth at least $g$ whose union is $G$. In this note, we give the $4$-girth-thickness $\theta(4,L(K_n))$ of the line graph of the complete graph $L(K_n)$ when $n$ is even. We also give the minimum number of subgraphs of $L(K_n)$, which are of girth at least $4$ and embeddable on the projective plane, whose union is $L(K_n)$.
\end{abstract}


	
\section{Introduction}
The \emph{thickness} $\theta(G)$ of a graph $G$ is the minimum number of elements in any partition of $E(G)$ such that the induced subgraph of each part is a planar graph. Equivalently, $\theta(G)$ is defined as the minimum number of planar subgraphs whose union is $G$.

The thickness has draw the attention of several researchers since its introduction in the 60s \cite{MR0157372} because it is an NP-hard problem \cite{MR684270} and it has many applications, for instance, in the design of circuits \cite{MR1079374}, in the Ringel's earth-moon problem \cite{MR1735339} and to bound the achromatic numbers of planar graphs \cite{araujo2017complete}, see the survey \cite{MR1617664}.

Only some exact results are known, for example, when $G$ is a complete graph \cite{MR0460162,MR0164339,MR0186573}, a hypercube \cite{MR0211901}, or a complete multipartite graph \cite{MR0158388,MR0229545,MR3243852,MR3610769}. And some generalizations of the thickness also have been studied such that the outerthickness $\theta_o$, defined similarly but with outerplanar instead of planar \cite{MR1100049}, and the $S$-thickness $\theta_S$, considering the thickness on a surfaces $S$ instead of the plane \cite{MR0245475}.

The \emph{$g$-girth-thickness} $\theta(g,G)$ of a graph $G$, introduced in \cite{rubio20174}, is the minimum number of elements in any partition of $E(G)$ such that the induced subgraphs of each part is a planar graph of girth at least $g$. The $g$-girth-thickness is the usual thickness when $g=3$ and it is the \emph{arboricity number} when $g=\infty$. Recall that the \emph{girth} of a graph is the size of its shortest cycle or $\infty$ if it is acyclic. 

Exact results also are known when $g>3$ and finite, for instance, the $4$-girth-thickness of the complete graph \cite{casta2017,GY,rubio20174}, the $4$-girth-thickness of the complete multipartite graph \cite{GY,rubio20175} and the $6$-girth-thickness of the complete graph \cite{casta2017}. Owing to the fact that the hypercube and the complete bipartite are triangle-free graphs, their thickness equal their 4-girth-thickness which were calculate in \cite{MR0211901} and partially calculate in \cite{MR0158388,MR0229545}, respectively.

We define the $S$-$g$-girth-thickness $\theta_S(g,G)$ of a graph $G$ as the minimum number of subgraphs embeddable on a surface $S$ of girth at least $g$ whose union is $G$. Of course, if $G$ has girth $g$ then $\theta_S(g,G)$ is $\theta_S(G)$ as in the case of $K_{n,n}$ for $g=4$, see \cite{MR0245475}.

In this note, we obtain the $4$-girth-thickness $\theta(4,L(K_n))$ of the line graph $L(K_n)$ of the complete graph $K_n$ when $n$ is even. To achieve this, in Section \ref{Section2} we recall some properties about token graphs $F_k(G)$. In Section \ref{Section3}, we determine $\theta(4,F_2(G))$ when $G$ contains a factorization into Hamiltonian paths, in particular \[\theta(4,L(K_n))=\frac{n}{2}\textrm{ and }\theta(4,F_2(K_{n-1,n}))=\frac{n}{2}\] for $n$ even. Finally, in Section \ref{Section4}, we determine $\theta_S(4,F_2(G))$ when $S$ is the projective plane and $G$ contains a Hamiltonian-factorization, in consequence \[\theta_S(4,L(K_n))=\left\lfloor \frac{n}{2}\right\rfloor\textrm{ and }\theta(4,F_2(K_{2n,2n}))=n\] for all $n$.


\section{Token graphs}\label{Section2}

Consider the following graph $F_k(G)$ called the \emph{$k$-token graph} introduced in \cite{MR2912660}, for given an integer $k\geq 1$ and a graph $G$ of order $n$. The vertex set $V(F_k(G))$ is the family of $k$ subsets of $V(G)$, therefore $|V(F_k(G))|=\binom{n}{k}$. Two such $k$-subsets $X$ and $Y$ are adjacent if its symmetric difference $X\triangle Y=\{x,y\}$ such that $x\in X$, $y\in Y$ and $xy\in E(G)$. The size of $F_k(G)$ is $\binom{n-2}{k-1}|E(G)|$, see \cite{MR2912660}. An example of a $2$-token graph is showed in Figure \ref{Fig1}, which is the $F_2(P_6)$.

\begin{figure}[htbp]
\begin{center}	
\includegraphics{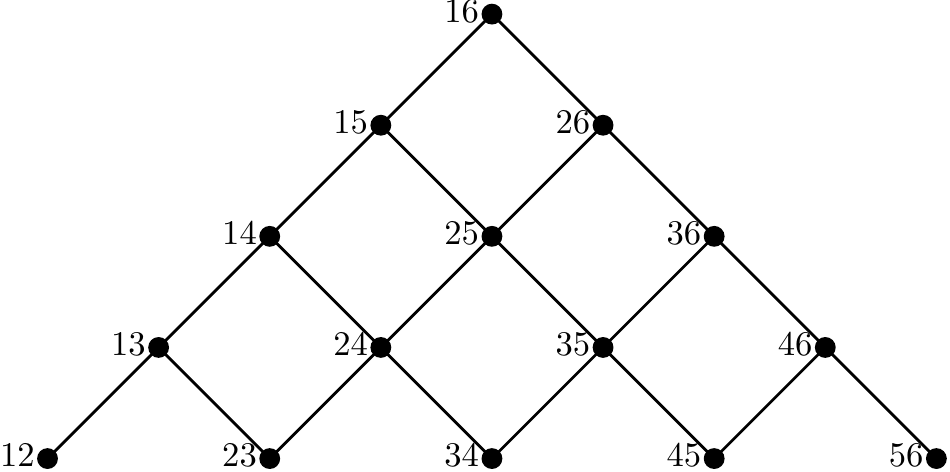}
\caption{\label{Fig1} The $2$-token graph of the path of order $6$.}
\end{center}
\end{figure}

The $2$-token graph $F_2(K_n)$ of the complete graph $K_n$ is the line graph $L(K_n)$ of the complete graph $K_n$ because each pair of incident edges $xz$ and $zy$ has symmetric difference the set $\{x,y\}$ which is the edge $xy$ of the complete graph. In general, the Johnson graph $J(n,k)\cong F_k(K_n)$ owing to the fact that it is the graph whose vertices are the $k$-subsets of an $n$-set, where two such subsets $X$ and $Y$ are adjacent whenever $|X\cap Y|= k-1$.

In \cite{MR3666665}, the authors remark that the $2$-token graph $F_2(P_n)$ of the path graph $P_n$ of $n$ vertices is planar of girth at least $4$ for every $n$.

Now, we prove that an edge-partition of a graph $G$ induces an edge partition of $F_k(G)$.

\begin{lemma}\label{lemma1}
Let $G$ be a non empty graph and $P=\{E_1,\dots,E_l\}$ an edge-partition of $G$. Then the set $\{E'_1,\dots,E'_l\}$ is an edge-partition of $F_k(G)$ where $E'_i=E(F_k(G[E_i]))$ for all $i\in\{1,\dots,l\}$.
\end{lemma}
\begin{proof}
Let $XY$ be an edge of $F_k(G)$, that is, $X$ and $Y$ are $k$-subsets of $V(G)$ such that for some $x\in X$ and some $y\in Y$, the symmetric difference of $X$ and $Y$ is $\{x,y\}$, and $xy$ is an edge of $G$. Let $j$ the unique index in the set $\{1,\dots,l\}$ such that $xy\in E_j$. Thus $xy$ is an edge of $G[E_j]$ and in consequence $XY$ is an edge of $F_k(G[E_j]) = E'_j$. Then $XY\in E'_j$. Moreover, if $XY \in E_i$ for some $i\in\{1,\dots,l\}$, then $xy\in E(G[E_i]) = E_i$. But $P = \{E_1,\dots,E_l\}$ is an edge-partition of $G$, and $xy\in E_j$, so $i=j$. Therefore each edge of $F_k(G)$ is in a unique element of $\{E'_1,\dots,E'_l\}$. In order to guarantee that every $E'_i$ is a non empty set, we need that $G$ has order at least $k+1$. In that case, if $xy\in E_i$ and $U = \{g_1,\dots,g_{k-1}\} \subseteq V(G)\setminus\{x,y\}$, then $X' = U \cup \{x\}$ and $Y' = U \cup \{y\}$ are two $k$-subsets of $G[E_i]$ such that its symmetric difference is $\{x,y\}$, and then $E'_i \neq \emptyset$, because $X'Y'\in E'_i$.
\end{proof}

\section{Determining $\theta (4,L(K_n))$ for $n$ even}\label{Section3}

A planar graph of $n$ vertices and girth at least $4$ has at most $2(n-2)$ edges for $n\geq 4$ and at most $n-1$, otherwise. In consequence, the $4$-girth-thickness $\theta(4,G)$ of a graph $G$ is at least $\left\lceil \frac{|E(G)|}{2(n-2)}\right\rceil$ for $n\geq 4$ and at least $\left\lceil \frac{|E(G)|}{n-1}\right\rceil$, otherwise.

Therefore we have the following theorem.

\begin{theorem}
If $G$ contains a factorization into $k$ Hamiltonian paths, then $\theta(4,F_2(G))=k$.
\end{theorem}
\begin{proof}
For $G=K_2$ or $G=P_3$, it is easy to check that $\theta(4,F_2(G))=1$. Assume that $G$ is a graph of order $n\geq 4$ containing a factorization into Hamiltonian paths. Then $G$ has size $e=(n-1)k\leq \binom{n}{2}$, then $k\leq n/2$ and \[k<n/2+1+1/(n-3).\]
Since, the $2$-token graph $F_2(G)$ has order $\binom{n}{2}$ and size $(n-2)(n-1)k$, it follows that 
\[\theta (4,F_2(G))\geq \left\lceil \frac{(n-2)(n-1)k}{2(\binom{n}{2}-2)}\right\rceil = \left\lceil k-\frac{2nk-6k}{n^2-n-4} \right\rceil.\]
Because $k<\frac{n}{2}+1+\frac{1}{n-3}=\frac{n^2-n-4}{2n-6}$ then \[0<\frac{k(2n-6)}{n^2-n-4}<1\]
and we have 
\[\theta (4,F_2(G))\geq k.\]
By Lemma \ref{lemma1}, the partition of $k$ Hamiltonian paths $\{G_1,\dots,G_k\}$ of $G$ induces a partition of $F_2(G)$ into $k$ planar subgraphs of girth at least 4, $\{F_2(G_1),\dots,F(G_k)\}$ and the result follows.
\end{proof}
We have the following corollaries.
\begin{corollary}
If $n$ is even then $\theta(4,F_2(K_{n-1,n}))=n/2$.
\end{corollary}
\begin{corollary}
If $n$ is even then $\theta(4,L(K_n))=n/2$.
\end{corollary}

\section{$\theta_S(4,L(K_n))$ when $S$ is the projective plane}\label{Section4}

Although the problem of finding the minimum number of planar graphs of girth at least $4$ into which the line graph of the complete graph can be decomposed remains partially solved, the corresponding problem can be solved for the surface called the projective plane. A similar proof provide the solution.

On one hand, a maximal graph of order $n$ and girth at least $4$ embeddable in the projective plane $S$ has size at most $2n-2$. On the other hand, since the $2$-token graph of a cycle is a graph embeddable in $S$ with girth $4$, see Figure \ref{Fig2} for a example, we can give the following theorem.

\begin{figure}[htbp]
\begin{center}	
\includegraphics{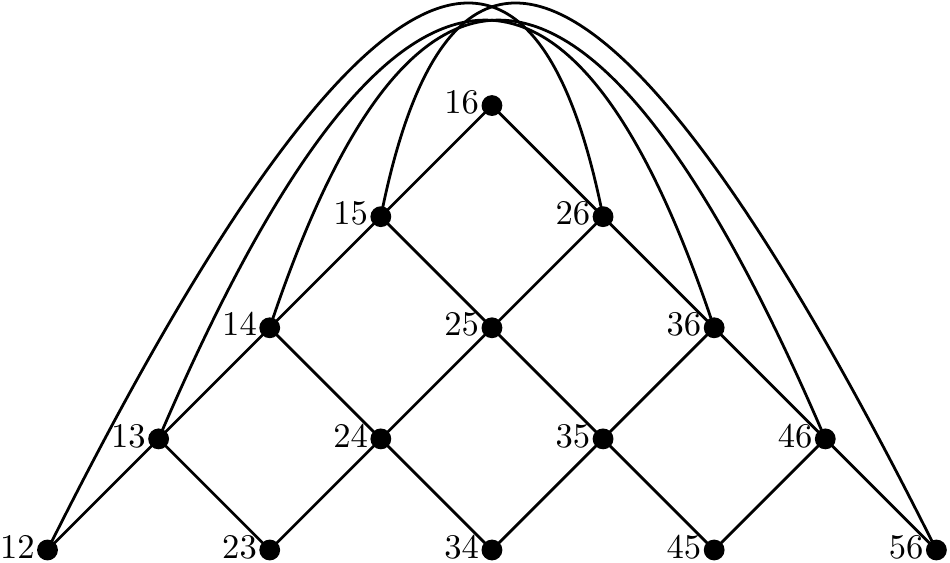}
\caption{\label{Fig2} The $2$-token graph of the cycle of order $6$.}
\end{center}
\end{figure}

\begin{theorem}
If $G$ is a graph of order $n\geq 4$ and contains a factorization into $k$ Hamiltonian cycles, then $\theta_S(4,F_2(G))=k$ when $S$ is the projective plane.
\end{theorem}
\begin{proof}
Let $G$ be a graph of order $n\geq 4$ containing a Hamiltonian-factorization, that is, a factorization into Hamiltonian cycles. Then $G$ has size $e=nk\leq \binom{n}{2}$, then $k\leq (n-1)/2$ and \[k<n+1+2/(n-2).\]
Since, the $2$-token graph $F_2(G)$ has order $\binom{n}{2}$ and size $(n-2)nk$, it follows that 
\[\theta_S(4,F_2(G))\geq \left\lceil \frac{(n-2)nk}{2\binom{n}{2}-2}\right\rceil = \left\lceil k-\frac{nk-2k}{n^2-n-2} \right\rceil.\]
Because $k<n+1+\frac{2}{n-2}=\frac{n^2-n-2}{n-2}$ then \[0<\frac{k(n-2)}{n^2-n-2}<1\]
and we have 
\[\theta_S (4,F_2(G))\geq k.\]
By Lemma \ref{lemma1}, the partition of $k$ Hamiltonian cycles $\{G_1,\dots,G_k\}$ of $G$ induces a partition of $F_2(G)$ into $k$ planar subgraphs of girth at least 4 embeddable in $S$, $\{F_2(G_1),\dots,F(G_k)\}$ and the result follows.
\end{proof}
We have the following corollaries.
\begin{corollary}
If $n$ is even then $\theta_S(4,F_2(K_{n,n}))=n/2$.
\end{corollary}
\begin{corollary}
For all $n$, we have that $\theta_S(4,L(K_n))=\left\lfloor \frac{n}{2}\right\rfloor$.
\end{corollary}

\section*{Acknowledgments}
Part of the work was done during the Reuni{\' o}n de Optimizaci{\' o}n, Matem{\' a}ticas y Algoritmos ROMA 2017, held at Casa Rafael Galv{\' a}n,  Universidad Aut{\' o}noma de Metropolitana, Mexico City, Mexico on July 24--28, 2017. 

The author wishes to thank F. Esteban Contreras-Mendoza for his useful discussions. 

Research partially supported by PAPIIT of Mexico grant IN107218.

\end{document}